\documentclass[11pt]{amsart}
\usepackage{amssymb, amsmath, amscd, dcpic, pictexwd}
\usepackage{hyperref}

\usepackage{amsopn}

\usepackage{tikz-cd}

\usepackage{mathtools}

\DeclareMathOperator{\End}{End}
\DeclareMathOperator{\Ker}{Ker}

\DeclareMathOperator{\Sym}{Sym}

\newtheorem{thm}{Theorem}[section]
\newtheorem{con}[thm]{Conjecture}

\newtheorem{cor}[thm]{Corollary}

\newtheorem{prop}[thm]{Proposition}
\newtheorem{defn}[thm]{Definition}
 
\newtheorem{exm}[thm]{Example}
\numberwithin{equation}{section}

\newcommand{\tY}{\tilde{Y}}

\newcommand{\cD}{\mathcal D}

\newcommand{\cO}{\mathcal{O}}

\newcommand{\cF}{\mathcal{F}}

\newcommand{\cN}{\mathcal N}
\newcommand{\cY}{\mathcal{Y}}

\newcommand{\fg}{\mathfrak{g}}

\newcommand{\fh}{\mathfrak{h}}

\newcommand{\CC}{\mathbb{C}}
\newcommand{\PP}{\mathbb{P}}

\newcommand{\LL}{\mathbb L}
\newcommand{\LLO}{\mathring{\LL}}
\newcommand{\LLP}{\LL^\perp}
\newcommand{\HH}{H}

\newcommand{\OX}{\mathcal{O}}
\newcommand{\R}{\mathcal{R}}
\newcommand{\ghat}{\hat{\fg}}

\newcommand{\tG}{\tilde{G}}

\newcommand{\tfg}{\tilde{\fg}}

\author{An Huang, Bong Lian, Shing-Tung Yau and Chenglong Yu}

\title[Tautological systems]{Period integrals of local complete intersections and tautological systems}
\begin{document}
\maketitle
\begin{abstract}
Tautological systems developed in \cite{LSY},\cite{LY} are Picard-Fuchs type systems to study period integrals of complete intersections in Fano varieties. We generalize tautological systems to local complete intersections, which are zero loci of global sections of vector bundles over Fano varieties. In particular, we obtain similar criterion as \cite{LSY, LY} about holonomicity and regularity of the system. We also prove solution rank formulas and geometric realizations of solutions following the work of hypersurfaces \cite{LHYZ, LHZ}. 
\end{abstract}

\tableofcontents 
\baselineskip=16pt plus 1pt minus 1pt
\parskip=\baselineskip

\section{Introduction}
Computing period integral has a long history in algebraic geometry. One way is to find enough linear differential operators that annihilate period integrals and study the corresponding Picard-Fuchs systems. Following this idea, tautological system is introduced in \cite{LSY}, \cite{LY}. It is a generalization of hypergeometric system by Gel'fand, Kapranov, and Zelevinski \cite{GKZ}.

Let $X^n/\mathbb{C}$ be a smooth $n$-dimensional Fano variety and $E$ be a vector bundle of rank $r$ over $X$. Denote the dual space of global sections by $V=\HH^0(X, E)^\vee$. Assume that any generic section $s\in V^\vee$ defines a nonsingular subvariety $Y_s=\{s=0\}$ in $X$ with codimension $r$. (When $E$ is very ample, the zero locus of a generic section is either empty or smooth due to a Bertini-type theorem for vector bundles proved by Cayley's trick. For example, see Lemma 1.6 in \cite{HS}. When it is empty, we can consider the quotient bundle of $E$ by the trivial line bundle.) The dimension of $Y_s$ is denoted by $d=n-r$. Consider the family of nonsingular varieties formed by zero loci of sections in $V^\vee$, denoted by $\pi \colon \cY\to B=V^\vee-D$, where $D$ is the discriminant locus. If we further assume $\det E\cong K_X^{-1}$, the adjunction formula implies that
\begin{equation}
\label{iso}
K_{Y_s}\cong K_X\otimes \det E|_{Y_s}\cong \cO_{Y_s}.
\end{equation}
A section $s$ of $K_X\otimes \det E\cong \mathcal{O}_X$ gives a family of holomorphic top forms $\Omega_s$ on $Y_s$ corresponding to constant section $1$ of $\cO_{Y_s}$, also called the residue of $s$.  We want to consider the period integral 
\begin{equation}
\int_\gamma \Omega_s, \quad\gamma\in \HH_{n}{(Y_s)}
\end{equation}
If $E$ splits as a direct sum of line bundles, the residue map is defined on line bundles and generalized to $E$ by induction. In the nonsplitting case, we can apply the residue formula in the splitting case locally and glue it together to get a global residue formula. The first isomorphism in (\ref{iso}) is induced by writing $s$ as line bundle sections and independent on the decomposition. This direct formula turns out to be hard for computations. Following the idea in \cite{LSY},\cite{LY}, we use the Calabi-Yau bundle structures to lift the bundles and sections to the principal bundles. Similar computations give us the differential relations from the symmetry of the bundles and geometric constrains from the defining ideal of $\PP(E^\vee)$ the projectivation of $E^\vee$ in $\PP(V)$. 

\section{Calabi-Yau bundles and adjunction formulas}
Motivated by the residue formulas for projective spaces and toric varieties, the notion of Calabi-Yau bundles is introduced in \cite{LY} and used to write down an adjunction formula on principal bundles. The canonical sections of holomorphic top forms used in period integral are given by this construction. First we recall the definition of Calabi-Yau bundles in \cite{LY} and adapted it to the local complete intersections. 
\begin{defn}[Calabi-Yau bundle]
Denote $H$ and $G$ to be complex Lie groups. Let $p\colon P\to X$ be a principal $H$-bundle with  $G$ equivariant action. A Calabi-Yau bundle structure on $(X, H)$ says that the canonical bundle of $X$ is the associated line bundle with character $\chi\colon H\to \mathbb{C}^*$. The following short exact sequence 
\begin{equation}
0\to \Ker p_*\to TP\to p^* TX\to 0
\end{equation}
induces an isomorphism
 \begin{equation}
 \label{residue}
K_P \cong p^* K_X \otimes  \det (P\times_{ad(H)}\fh^\vee).
\end{equation}
Fixing an isomorphism $P\times_H \CC_\chi\cong K_X$, the isomorphism (\ref{residue}) implies that $K_P$ is a trivial bundle on $P$ and has a section $\nu$ which is the tensor product of nonzero elements in $\CC_\chi$ and $\det \fh^\vee$. This holomorphic top form satisfies that 
\begin{equation}
\label{nu}
h^*(\nu)=\chi(h)\chi_\fh^{-1} (h) \nu,
\end{equation}
where $\chi_\fh$ is the character of $H$ on $\det \fh$ by adjoint action. The tuple $(P, H, \nu, \chi)$ satisfying (\ref{nu}) is called a Calabi-Yau bundle. 
\end{defn}
Conversely, any section $\nu$ satisfying (\ref{nu}) determines an isomorphism $P\times_H \CC_\chi\cong K_X$. Since the only line bundle automorphism of $K_X\to X$ fixing $X$ is rescaling when $X$ is compact, such $\nu$ is determined up to rescaling (Theorem 3.12 in \cite{LY}). So the equivariant action of $G$ on $P\to X$ changes $\nu$ according to a character $\beta^{-1}$ of $G$. We say the Calabi-Yau bundle is $(G, \beta^{-1})$-equivariant.

\begin{exm}
\label{CPn}
Let $X$ be $\CC\PP ^{d+1}$ and $P$ be $\CC^{d+2}\setminus\{0\}$ with nature actions of $G=GL(d+2, \CC)$ and $H=\CC^*$. The volume form is $\nu=dx_0\wedge \cdots \wedge dx_{d+1}$. The character $\beta^{-1}=\det g$ for any $g\in G$.
\end{exm}

When $E$ is the line bundle $K_X^{-1}$, the following is the residue formula for Calabi-Yau bundles:

\begin{thm}[\cite{LY}, Theorem 4.1]
If $(P, H, \nu, \chi)$ is a Calabi-Yau bundle over a Fano manifold $X$, the middle dimensional variation of Hodge structure $R^d \pi_*(\underline{\CC})$ associated with the family $\pi\colon \cY\to B$ of Calabi-Yau hypersurfaces has a canonical section of the form 
\begin{equation}
\label{residue1}
\omega =Res \frac{\iota_{\xi_1}\cdots \iota_{\xi_m}\nu}{f}.
\end{equation}
Here ${\xi_1},\cdots ,{\xi_m}$ are independent vector fields generating the distribution of $H$-action on $P$, and $f\colon B\times P\to \CC$ is the function representing the universal section of $P\times_H \CC_{\chi^-1}\cong K_X^{-1}$.
\end{thm}
When $E$ is a direct sum of line bundles associated to characters of $H$, the residue formula is similar to (\ref{residue1}) by induction.
\begin{defn}
\label{residue0}
When $E$ is a vector bundle associated with representation $\rho\colon H\to GL(W)$, we can also construct a residue formula as follows. Under the assumption that $E=P\times_H W$, a sections of $E$ is a $(H,\rho)$-equivariant map $f\colon P\to W$, i.e. $f(p\cdot h)=h\cdot f(p)$. Choose a basis $e_1,\cdots, e_r$ for $W$, then $f=f^1e_1+\cdots+f^re_r$. Assume $f$ defines a smooth Calabi-Yau subvariety $Y_f$ with codimsion $r$ in $X$. We have the following residue formula:
\begin{equation}
\label{residue2}
\omega = \iota_{\xi_1}\cdots \iota_{\xi_m} Res \frac{\nu}{f^1\cdots f^r}
\end{equation}
The residue defines a holomorphic top form on the zero locus of $f^i=0$ on $P$, which is the restriction of the principal bundle on the Calabi-Yau subvariety. After contracting with ${\xi_1},\cdots ,{\xi_m}$, the holomorphic $d$-form $w$ is invariant under the action of $H$ and vanishes for the vertical distribution, hence it defines a $d$-form on $Y_f$.
\end{defn}
The vector bundle $E$ can be associated with different principal bundles. The residue construction is canonical in the following sense.
\begin{prop}
\label{uniqueness}
For different choice of principal bundles $P$ or basis $e_1,\cdots, e_r$, the residue form $\omega$ is unique up to a scalar which is independent of $f$.
\end{prop}
\begin{proof}
Firstly, for different choices of basis of $W$, the functions $f_i$ are changed by a linear transformation. Hence the denominator of the residue formula is changed by the determinant of the linear transformation along the common zero locus of $f_i$. So the residue is changed by a scalar.

Secondly, we prove the independence on the choice of principal bundles. Let $P^\prime$ be the frame bundle of $E$. We only need to compare any $P$ with $P^\prime$. The frame bundle can be constructed by $P^\prime=P\times_H GL(W)$. So we have a quotient map from $P\times GL(W)$ to $P^\prime$. Especially we have a principal bundle map 
\begin{equation}
c\colon P\to P^\prime,
\end{equation}
which is equivariant under the actions of $H$ and $GL(W)$ on both sides related by $\rho\colon H\to GL(W)$. Then we have a morphism between the exact sequence
\begin{equation}
\begin{tikzcd}
0\arrow{r}&\Ker p_*\arrow{r}\arrow{d} &TP\arrow{r}\arrow{d} &p^* TX\arrow{r}\arrow{d} &0 \\
0\arrow{r}&\Ker {p^\prime}_*\arrow{r} &T{P^\prime} \arrow{r} &p{^\prime}^* TX\arrow{r} &0
\end{tikzcd}
\end{equation}
Notice that one exact sequence gives an isomorphism $p^*K_X \cong  K_P \otimes  \det (P\times_{ad(H)}\fh)$ and induces the form $\nu$ and $\iota_{\xi_1}\cdots \iota_{\xi_m} \nu$. So we have 
\begin{equation}
\iota_{\xi_1}\cdots \iota_{\xi_m} \nu=c^*(\iota_{\xi_1^\prime}\cdots \iota_{\xi_r^\prime} \nu^\prime).
\end{equation}
Furthermore, the functions $f_i$ defined on $P$ are also pull back of the corresponding functions on $P^\prime$. So we have the same residue formula.
\end{proof}
With the canonical choice of $\omega$, the period integral for the family is defined to be
\begin{equation}
\Pi_\gamma=\int_\gamma \omega.
\end{equation}
Here $\gamma$ is a local horizontal section of the $d$-th homology group of the family. The period integrals are local holomorphic functions on the base $B$ and generates a subsheaf of $\mathcal{O}_B$ called period sheaf.

\section{Tautological systems}
In order to study the period sheaf, we look for differential operators which annihilate the period integrals. In \cite{LSY} and \cite{LY}, tautological systems are the introduced and the solution sheaves contain the period sheaves. When $H$ is the complex torus and $X$ toric variety, tautological systems are the GKZ systems and extended GKZ systems. When $X$ is homogenous variety, tautological systems provide new interesting $\mathcal{D}$-modules. The notion of tautological system also provides convenient ways to apply $\mathcal{D}$-module theory to study the solution sheaves and period sheaves. The regularity and holonomicity are discussed in \cite{LSY}, \cite{LY}. The Riemann-Hilbert problems and geometric realizations are discussed in \cite{BHLSY}, \cite{LHZ}, \cite{LHYZ}. 

The differential operators in tautological systems come from two sources: one from symmetry group $G$ called symmetry operators and the other from the defining ideal of $X$ in the linear system $|K_X^{-1}|$ called geometric constrains. In this section we have similar constructions. First we fix a basis $a_1, \cdots, a_m$ of $V$ and dual basis $a_1^\vee, \cdots, a_m^\vee$ of $V^\vee$. Viewing $a_i$ as coordinates on $V^\vee$, the universal section of $E$ is denoted by $f=a_1a_1^\vee+\cdots+a_ma_m^\vee$. According to the discussion in last section, the section $a_i^\vee$ corresponds to a map $f_i\colon P\to W$ and a tuple of functions $f_i=(f_i^1,\cdots,f_i^r)$. Then the residue formula has the following form:
\begin{equation}
\label{residue3}
\omega = \iota_{\xi_1}\cdots \iota_{\xi_m} Res \frac{\nu}{(\sum_i a_i f_i^1)\cdots (\sum_i a_i f_i^r)}.
\end{equation}

Considering the action of $G$ on $V$, we have a Lie algebra representation 
\begin{equation}
Z\colon \fg \to \End V
\end{equation}
For any $x\in \fg$, we denote $Z(x)=\sum_{i,j} x_{ji}a_ja_i^\vee$ and the dual representation $Z^\vee(x)=\sum_{i,j}-x_{ij}a_i^\vee a_j$.

From Proposition \ref{uniqueness}, the uniqueness of residue formula, we know that the $G$ action changes period integral according to a character of $G$. So the first order differential operators $\sum_{ij}x_{ij}a_i\frac{\partial}{\partial{a_j}}+\beta(x)$ annihilate the period integral. More specifically, consider the action of the one parameter group $\exp (tx)$ acting on the period integral.
From Cartan's formula, 
\begin{equation}
L_x \iota_{\xi_1}\cdots \iota_{\xi_m} \frac{\nu}{(\sum_i a_i f_i^1)\cdots (\sum_i a_i f_i^r)}=d (\iota_x \iota_{\xi_1}\cdots \iota_{\xi_m} \frac{\nu}{(\sum_i a_i f_i^1)\cdots (\sum_i a_i f_i^r)}).
\end{equation}
So we have 
\begin{equation}
\int_\gamma \iota_{\xi_1}\cdots \iota_{\xi_m} Res (L_x \frac{\nu}{(\sum_i a_i f_i^1)\cdots (\sum_i a_i f_i^r)})=0.
\end{equation}
The Lie derivative is 
\begin{align}
&L_x \frac{\nu}{(\sum_i a_i f_i^1)\cdots (\sum_i a_i f_i^r)}\\&=-\beta(x) \frac{\nu}{(\sum_i a_i f_i^1)\cdots (\sum_i a_i f_i^r)}+\sum_k\frac{\sum_{i,j}a_ix_{ij}f_j^k\nu}{(\sum_i a_i f_i^1)\cdots(\sum_i a_i f_i^k)^2\cdots (\sum_i a_i f_i^r)}\\
&=(-\beta(x)-\sum_{ij}x_{ij}a_i\frac{\partial}{\partial{a_j}})\frac{\nu}{(\sum_i a_i f_i^1)\cdots (\sum_i a_i f_i^r)}.
\end{align}
So we have
\begin{equation}
(\beta(x)+\sum_{ij}x_{ij}a_i\frac{\partial}{\partial{a_j}})\Pi_\gamma=0.
\end{equation}

The geometric constrains arise from the following observation. Consider the first order differential operators:
\begin{equation}
\frac{\partial}{\partial{a_j}}  \frac{\nu}{(\sum_i a_i f_i^1)\cdots (\sum_i a_i f_i^r)}=-\sum_k\frac{f_j^k\nu}{(\sum_i a_i f_i^1)\cdots(\sum_i a_i f_i^k)^2\cdots (\sum_i a_i f_i^r)};
\end{equation}
the second order differential operators:
\begin{equation}
\frac{\partial^2}{\partial{a_l}\partial{a_j}}  \frac{\nu}{f^1\cdots f^r}=\sum_{a,b}P_{ab}f_l^af_j^b \nu
\end{equation}
Here $P_{ab}$ are rational functions of $f_1,\cdots, f_r$, not depending on $j,l$. By induction, we will have similar formula for higher order differential operators $\partial_{i_1,\cdots, i_s}$ and the $P$ coefficients are independent of the multi-index ${i_1,\cdots, i_s}$.Notice that we can switch the order of $l$ and $j$. So we have $P_{ab}=P_{ba}$. On the other hand, consider the product of $a_l^\vee$ and $a_j^\vee$ in $\HH ^0(X, \Sym^2 E)$. The symmetric product $\Sym^2 E$ is associated with the symmetric product of the representation of $\rho$. Hence $a_l^\vee\cdot a_k^\vee$ can be viewed as a map $P\to \Sym^2 W$ 
\begin{equation}
(\sum_af_l^ae_a)\cdot(\sum_bf_j^ae_b)=\sum_a f_l^af_j^a e_a^2+
\sum_{a<b} (f_l^af_j^b+f_j^af_l^b) e_ae_b.
\end{equation}
Consider the elements in the kernel of the map 
\begin{equation}
\HH ^0(X, E)\otimes \HH ^0(X, E)\to \HH ^0(X, \Sym^2 E).
\end{equation}
The Fourier transform of these elements annihilate period integral.
For example if $(a_1^\vee)^2-a_2^\vee a_3^\vee=0$, then
\begin{equation}
(\frac{\partial^2}{\partial{a_1}^2}-\frac{\partial^2}{\partial{a_2}\partial{a_3}})\frac{\nu}{f^1\cdots f^r}=(\sum_a P_{aa} ((f_1^a)^2-f_2^af_3^a)+\sum_{a<b} P_{ab}(2f_1^af_1^b+f_2^af_3^b+f_3^af_2^b))\nu=0.
\end{equation}
This is because the terms of $f_i^a$ are the coefficients of $(a_1^\vee)^2-a_2^\vee a_3^\vee$ written under the basis $e_a^2, e_ae_b$.

In order to describe the geometric origin of the differential operators above, we need the well-known facts relating the vector bundle $E$ and hyperplane line bundle $\cO(1)$ on $\PP(E^\vee)$.
\begin{prop}
Assume $E\to X$ is a holomorphic vector bundle on complex manifold $X$ and $\cO(1)\to \PP(E^\vee)$ is the hyperplane bundle on the projectivation of $E^\vee$. There is a canonical ring isomorphism 
\begin{equation}
\oplus_k\HH^0(X, \Sym^k(E))\cong \oplus_k\HH^0(\PP(E^\vee), \cO(k)).
\end{equation}
\end{prop}
\begin{proof}
It follows from Leray-Hirsch spectral sequence computing $\HH^0(\PP(E^\vee), \cO(k))$.
\end{proof}
The above identification of $V^\vee$ with $\HH^0(\PP(E^\vee), \cO(1))$ gives a map $\PP(E^\vee)\to \PP(V)$ by $\cO(1)$ when $|\cO(1)|$ is base-point free. Consider the ideal determined by image of this map $I(\PP(E^\vee), V)$, which is the kernel of the map 
\begin{equation}
\oplus_k \Sym^k(V^\vee)\to \oplus_k\HH^0(X, \Sym^k(E))
\end{equation}

With the discussion above, we collect all the differential operators in the following theorem.
\begin{thm}
\label{tauto}
The period integral $\Pi_\gamma$ satisfies the following system of differential equations:
\begin{align}
&Q(\partial_a)\Pi_\gamma=0\quad\quad (Q\in I(\PP(E^\vee), V))\\
&(Z_x +\beta(x))\Pi_\gamma=0 \quad\quad (x\in \fg)\\
&(\sum_i a_i\frac{\partial}{\partial a_i}+r)\Pi_\gamma=0.
\end{align}
\end{thm}
The last operator is called Euler operator and comes from $\omega$ being homogenous of degree $-r$ with respect to $a_i$. We can also view Euler operator as symmetry operator. Consider the frame bundle of $E$ with structure group $H=GL(r)$.  It has a symmetry $G=\CC^*$ acting as the center of $H$. The symmetry operator of $G$ is the Euler operator. 

We call the differential system in Theorem \ref{tauto} tautological system for $(X, E, H, G)$. It's the same as the cyclic $D$-module $\tau(G, \PP(E^\vee), \cO(-1), \hat{\beta})$ defined in \cite{LSY} \cite{LY} by  
\begin{equation}
\tau=D_{V^\vee}/D_{V^\vee}(J(\PP(E^\vee))+Z(x)+\hat{\beta}(x), x\in \hat{\fg}).
\end{equation}
Here $J(\PP(E^\vee))=\{Q(\partial_a)|Q\in I(\PP(E^\vee))\}$, $\hat{G}=G\times \CC^*$ with Lie algebra $\hat{\fg}=\fg\oplus \CC e$ and $\hat{\beta}=(\beta, r)$.

We can apply the holonomicity criterion for tautological system in \cite{LSY},\cite{LY}.

\begin{thm}
If the induced action of $G$ on $\PP(E^\vee)$ has finite orbits, the corresponding tautological system $\tau$ is regular holonomic. 
\end{thm}

\section{Examples}
\begin{exm}[Complete intersections]
\label{ci}
When $E=\oplus^r_1 L_i$ is a direct sum of very ample line bundles, the above system recovers the tautological system for complete intersections in \cite{LY}. This case is equivalent to say that the structure group of $E$ is reduced to the complex torus $(\CC^*)^r$. So we have symmetry group $(\CC^*)^r$ acting on the fibers of $E$. This gives the usual Euler operators in \cite{LY}. Let $\hat{X}_i$ be the cone of $X$ inside $V_i=\HH^0(X, L_i)^\vee$ under the linear system of $L_i$. The cone of $\PP(E^\vee)$ inside $V=\oplus_{i=1}^r V_i$ is fibered product $\hat{X}_i$ over $X$. So the geometric constrains are the same as \cite{LY}. Assume $X$ is a $G$-variety consisting of finite $G$-orbits and $L_i$ are $G$-homogenous bundles. Then $\PP(E^\vee)$ admits an action of $\tG=G\times (\CC^*)^{r-1}$ with finite orbits. This is the same holonomicity criterion as \cite{LY} for complete intersections.
\end{exm}

\begin{exm}[Homogeneous varieties]
Let $G$ be a semisimple complex Lie group and $X=G/P$ is a generalized flag variety quotient by a parabolic subgroup $P$. This forms a principal $P$-bundle over $X$. We assume $E$ to be a homogenous vector bundle from a representation of $P$ and the action of $G$ on $\PP(E^\vee)$ is transitive. Then the projectivation of $\PP(E^\vee)$ is also a generalized flag variety for a parabolic subgroup $P'\subset P$. Hence the $G$-action on $\PP(E^\vee)$ is transitive. If $\cO(1)$ is very ample on $\PP(E^\vee)$, the defining ideal of  $\PP(E^\vee)$ in $\PP(V)$ is given by the Kostant-Lichtenstein quadratic relations. Furthermore, any character of $G$ is trivial, hence $\beta$ is zero. So the differential system is regular holonomic and explicitly given in this case.
\end{exm}

\section{Global residue on $\PP(E^\vee)$}
The residue formula in Definition \ref{residue0} is motivated by residue formulas for hypersurfaces. It is locally the same as complete intersections. In this section, we introduce another approach more directly related to the geometry of $\PP(E^\vee)$ by Cayley's trick. 

First we fix some notations. Let $\PP=\PP(E^\vee)$ and $\OX(1)$ for the hyperplane section bundle. The projection map is denoted by $\pi\colon \PP\to X$. Any section $f\in H^0(X, E)$ is identified with a section $f\in H^0(\PP,\OX(1))$. The zero locus of $f\in H^0(\PP,\OX(1))$ is denoted by $\tY_f$. 

We collect the propositions relating the geometry of $X$ and $\PP$ in the following.

\begin{prop}

\begin{enumerate}
\item Hypersruface $\tY_f$ is smooth if and only if $Y_f$ is smooth with codimension $r$ or empty. 
\item There is an natural isomorphism $K_\PP\cong \pi^*(K_X\otimes \det E)\otimes \OX(-r)$
\end{enumerate}
\end{prop}
\begin{proof}
The proof of first two propositions are the same as toric complete intersections \cite{Mavlyutov}.
\end{proof}

From now on, we assume $Y_f$ is smooth with codimesion $r$. 
\begin{defn}
The variable cohomology $H^{d}_{var}(Y_f)$ is defined to be cokernel of $H^{d}(X)\to H^{d}(Y_f)$. 
\end{defn}

Using the same argument for toric complete intersections in \cite{Mavlyutov}, we have the following propositions
\begin{prop}
There is an long exact sequence of mixed hodge structures
\begin{equation}
0\to H^{n-r-1}(X)\to H^{n+r-1}(X)\to H^{n+r-1}(X-Y_f)\to H^{d}_{var}(Y_f)\otimes \CC[r]\to 0
\end{equation}
Here $\CC[r]$ is the $r$-th Tate twist.
\end{prop}

\begin{prop}
The map $\pi\colon \PP-\tY_f\to X-Y_f$ is fibration with fibers isomorphic to $\CC^{r-1}$ and induces an isomorphism of mixed Hodge structures $\pi^*\colon H^{n+r-1}(X-Y_f)\to H^{n+r-1}(\PP-\tY_f)$.
\end{prop}

We now consider the Calabi-Yau case, equivalently $\det E\cong K_X^{-1}$, for simplicity. Then we have the vanishings in the Hodge filtration $F^{n+r-1-k}=0$ for $k<r-1$ and isomorphisms
\begin{equation}
H^0(\PP, \OX_\PP)\to H^0(\PP, K_\PP\otimes \OX(r))\to F^{n}H^{n+r-1}(\PP-\tY_f)\to F^{n-r}H^{n-r}(Y_f).
\end{equation}
\begin{prop}
The constant function $1$ is sent to holomorphic top form $\omega_f$ on $Y_f$ via this sequence of isomorphisms. Then $\omega_f$ is the same as $\omega$ in Definition \ref{residue0}. 
\end{prop}

Consider the principle bundle adjunction formula for base space $\PP$. Let $(P, H, \nu, \chi)$ is a Calabi-Yau bundle over $\PP$. The image of $1$ in $H^0(\PP, K_\PP\otimes \OX(r))$ has the form $\Omega\over f^r$ on principle bundle $P$. If we write $f$ as universal section, then similar calculation can recover the differential operators in Theorem \ref{tauto}.

\section{Solution rank}
Now we discuss the solution rank for the system. There are two versions of solution rank formula for hypersurfaces. One is in terms of Lie algebra homology, see \cite{BHLSY}. One is in terms of perverse sheaves on $X$, see \cite{LHZ}. Here we have similar description for zero loci of vector bundle sections.
\subsection{Lie algebra homology description}
We fix some notations. Let $R=\CC[V]/I(\PP(E^\vee))$ be the coordinate ring of $\PP$. Let ${Z}\colon \ghat\to \End(V)$ be the extended representation by $e$ acting as identity. We extend the character $\beta\colon \ghat\to \CC$ by assigning $\beta(e)=r$.
\begin{defn}
We define $\cD_{V^\vee}$-module structure on $R[a]e^{f}\cong R[a_1,\cdots a_N]$ as follows. The functions $a_i$ acts as left multiplication on $R[a_1,\cdots a_N]$. The action of $\partial_{a_i}$ on $R[a_1,\cdots a_N]$ is $\partial_{a_i}+a_i^\vee$.
\end{defn}

Then we have the following $\cD_{V^\vee}$-module isomorphism. 
\begin{thm}
\label{Lietau}
There is a canonical isomorphism of $\cD_{V^\vee}$-module
\begin{equation}
\tau\cong R[a]e^{f}/{Z}^\vee(\ghat) R[a]e^{f}
\end{equation}
\end{thm}
This leads to the Lie algebra homology description of (classical) solution sheaf
\begin{cor}
If the action of $G$ on $\PP(E^\vee)$ has finitely many orbits, then the stalk of the solution sheaf at $b\in V^\vee$ is
\begin{equation}
sol(\tau)\cong Hom_\cD(R e^{f(b)}/{Z}^\vee(\ghat) Re^{f(b)}, \OX_b)\cong H_0(\ghat, R e^{f(b)})
\end{equation}
\end{cor}

\subsection{Perverse sheaves description}
We follow the notations in \cite{LHZ}.
\begin{enumerate}
\item Let $\LL^\vee$ be the total space of $\OX(1)$ and $\LLO^\vee$ the complement of the zero section.
\item Let $ev\colon V^\vee\times \PP\to \LL^\vee, \quad (a, x)\mapsto a(x)$ be the evaluation map. 
\item Assume $\LLP=\ker (ev)$ and $U=V^\vee\times \PP-\LLP$. Let $\pi\colon U\to V^\vee$. Notice that $U$ is the complement of the zero locus of the universal section. 
\item Let $\cD_{\PP,\beta}=(\cD_\PP\otimes k_\beta)\otimes_{U\fg} k$, where $k_\beta$ is the $1$-dimensional $\fg$-module with character $\beta$ and $k$ is the trivial $\fg$-character.
\item Let $\cN=\OX_{V^\vee}\boxtimes \cD_{\PP,\beta}$.
\end{enumerate}
We have the following theorem. See Theorem 2.1 in \cite{LHZ}.
\begin{thm}
There is a canonical isomorphism
\begin{equation}
\tau\cong H^0\pi_{+}^\vee(\cN |_U)
\end{equation}
\end{thm}

\begin{proof}
The proof follows the same arguments of Theorem 2.1 in \cite{LHZ}. The only difference is that the universal section $f=a_i\otimes a_i^\vee$ defines a trivialization of $\OX_{V^\vee}\boxtimes \OX(1)$. So $f^{-r}$ instead of $f^{-1}$ defines a nonzero section of $\OX_{V^\vee}\boxtimes \omega_\PP$. Hence we have isomorphism 
\begin{equation}
\OX_U f^{-r}\cong \omega_{U/V^\vee}.
\end{equation}
Let $\R\coloneqq \OX_{V^\vee}/\OX_{V^\vee} J(\PP)$ be a $\cD_{V^\vee}\times \ghat$-module. Then from Theorem \ref{Lietau}, we have 
\begin{equation}
\tau\cong (\R\otimes k_\beta)\otimes_{\ghat} k.
\end{equation}
The $\cD_{V^\vee}\times \ghat$-morphism in the technical Lemma 2.6 is now changed to 
\begin{equation}
\phi\colon \R \otimes k_\beta\to \OX_U f^{-r}
\end{equation}
by setting 
\begin{equation}
\phi(a\otimes b)={(-1)^l (l+r)!\over f^{l+r}}a\otimes b
\end{equation}
Here we identify $\R$ with $\OX_{V^\vee}\otimes \mathcal{S}$ and $\mathcal{S}$ is the graded coordinate ring of $(\PP, \OX(1))$. The element $b\in \mathcal{S}$ has degree $l$.
Since $\beta(e)=r$, we have an isomorphism induced by $\phi$
\begin{equation}
\tau\cong (\R\otimes k_\beta)\otimes_{\ghat} k\cong(\OX_U f^{-r})\otimes_{\fg} k.
\end{equation}
\end{proof}
A direct corollary is the following
\begin{cor}
If $\beta(\fg)=0$, there is a canonical surjective map 
\begin{equation}
\tau\to H^0\pi_+^\vee\OX_U.
\end{equation}
In terms of period integral, we have an injective map
\begin{equation}
H_{n+r-1}(X-Y_b)\to Hom(\tau, \OX_{V^\vee, b}) 
\end{equation}
given by 
\begin{equation}
\gamma\mapsto \int_\gamma {\Omega\over f^r}
\end{equation}
\end{cor}

We have similar solution rank formula. We assume $G$-action on $\PP$ has finitely many orbits. Let $\cF=Sol(\cD_{\PP,\beta})=RHom_{\cD^{an}}(\cD_{\PP,\beta}^{an}, \OX_\PP^{an})$ be a perverse sheaf on $\PP$.
\begin{cor}
Let $b\in V^\vee$. Then the solution rank of $\tau$ at $b$ is $\dim H^0_c(U_b, \cF|_{U_b})$.
\end{cor}
Now we apply the solution rank formulas to different cases.
\subsection{Irreducible homogeneous vector bundles}
In this subsection, we assume $X$ is homogeneous $G$-variety and the lifted $G$-action on $\PP$ is also transitive. In other words, we have $X=G/P$ and $\PP=G/P^\prime$ with $P/P^\prime\cong \PP^{r-1}$. Then we have the following corollary
\begin{cor}
If $\beta(\fg)=0$, then the solution rank of $\tau$ at point $b\in V^\vee$ is given by $\dim H_{n+r-1}(X-Y_b)$.
\end{cor}
\begin{proof}
The solution sheaf in this case is $\cF\cong \CC[n+r-1]$. So the solution rank is 
\begin{equation}
\dim H^0_c(U_b, \CC[n+r-1])=\dim H_{n+r-1}(U_b)=\dim H_{n+r-1}(X-Y_b)
\end{equation}
\end{proof}

\begin{exm}
Let $X=G(k, l)$ be Grassmannian and $F$ be the tautological bundle of rank $k$. Then $E\cong F^\vee\otimes \OX({l-1\over k})$ is an ample vector bundle with $\det E\cong K_X^{-1}$. The corresponding $\PP$ is homogenous under the action of $SL(l+1)$
\end{exm}

\subsection{Complete intersections}
We first fix some assumptions. 
\begin{enumerate}
\item Let $E$ split as direct sum of homogenous $G$-line bundles $L_1,\cdots, L_r$. 
\item Let $\tG=G\times (\CC^*)^{r-1}$ acting on $\PP$ as Example \ref{ci}. 
\item Let $G$-action of $X$ have finitely many orbits. Then $\tG$-action on $\PP$ has finitely many orbits. 
\item We further assume $\beta(\fg)=0$. This implies $\beta(\tfg)=0$. In this case, there are always some invariant divisor on $\PP$. 
\end{enumerate}
Let $[t_1, \cdots, t_r]$ be the local homogenous coordinates on $\PP$ in $\PP^{r-1}$-direction. Each $t_i$ comes from the coordinate on $L_i^\vee$. Then $t_i=0$ defines globally a divisor on $\PP$. We denote it by $D_i$. The complement of $\cup_i D_i$ is denoted by $\mathring{\PP}$. Let $j\colon \mathring{\PP}\to \PP$ be the open embedding. We treat the following two special cases with $X$ being homogenous or toric. 
 
 Let $X=G/P$ be a homogenous $G$-variety. Then we have the isomorphism
 \begin{prop}
 $\cD_{\PP, \beta}\cong j_!j^!\cD_{\PP, \beta}$
 \end{prop}
 \begin{proof}
 The proof follows from the proof in \cite{LHYZ}. The key observation is that Lemma 3.2 and Corollary 3.3 only uses the condition that toric divisors have normal crossing singularities and the ambient space is a log homogenous variety with respect to these divisors. See the section about chain integral in log homogenous varieties.
 \end{proof}
 So we have the following description of solution rank 
 \begin{thm}
 The solution rank at point $b\in V^\vee$ is given by 
 \begin{equation}
 H_{n+r-1}(\PP-\tY_b, (\PP-\tY_b)\cap (\cup_i D_i))
 \end{equation}
 \end{thm}
 This theorem is not satisfying because the final cohomology is not directly related to $X$. Let $Y_1, \cdots, Y_r$ be the zero locus of the $L_i$ component of section $s_b$. From the geometric realization of some solutions as period integral as rational forms along the cycles in the complement of $Y_1\cup \cdots\cup Y_r$, we have the following conjecture:
 \begin{con}
 There is a natural isomorphism of solution sheave as period integrals
 \begin{equation}
Hom_{\cD_{V^\vee}}(\tau_{V}, \OX)|_b\cong H_n(X-(Y_1\cup \cdots\cup Y_r))
\end{equation}
 \end{con}

\section{Chain integrals in log homogenous varieties}
Let $X$ be a complex variety with normal crossing divisor $D$. The log $D$ tangent bundle $T_X(-\log D)$ is a subsheaf of $T_X$ defined as follows. If $x_1, \cdots, x_n$ is the local coordinate of $X$ and $D$ is the hyperplanes defined by $z_1=0,\cdots, z_r=0$, then the generating sections of $T_X(-\log D)$ are $z_1\partial_1,\cdots, z_r\partial_r, \partial z_{r+1}, \cdots, z_n$. Then we say $X$ is log homogenous if $T_X(-\log D)$ is globally generated. Let $\fg$ be $H^0(X, T_X(-\log D)$ and $G$ be the corresponding Lie group. Then $X$ is an $G$-variety. Let $L$ be a $G$-equivariant line bundle defined on $X$. We assume $L$ is very ample and $L+K_X$ is base point free. Then the period integrals of sections in $W^\vee=H^0(X, L+K_X)$on hypersurface families cut out by $V^\vee=H^0(X, L)$ satisfies the tautological systems $\tau_{VW}$ with character $\beta=0$. See \cite{LY} and \cite{LHZ}.  Then $\tau$ is holonomic because $G$-action on $X$ is stratified by $D$ with finite orbits. See \cite{brion2006log} for discussion of log homogenous varieties.

We consider the solution rank of $\tau_{VW}$ in this case. Following the same proof in \cite{LHYZ}, we have the following description of solution rank of $\tau_{VW}$.
\begin{thm}
There is an natural isomorphism 
\begin{equation}
Hom_{\cD_{V^\vee\times W^\vee}}(\tau_{VW}, \OX)|(a, b)\cong H_n(X-Y_a, (X-Y_a)\cap( \cup D))
\end{equation}
given by period integral.
\end{thm}



\bibliography{reference1.bib}
\bibliographystyle{abbrv}


\end{document}